\documentclass{amsart}
\usepackage{amsmath}
\usepackage{amssymb}
\usepackage{mathrsfs}
\usepackage{stmaryrd}
\usepackage{mathtools}
\usepackage{nicematrix}
\usepackage{tikz}
\usepackage[all]{xy}
\usepackage[utf8]{inputenc}
\usepackage{enumerate}
\usepackage{a4wide}
\usepackage{tikz-cd}
\usepackage{hyperref}

\newtheorem{thmvoid}{}[section]
\newtheorem{theorem}[thmvoid]{Theorem}
\newtheorem{corollary}[thmvoid]{Corollary}
\newtheorem{lemma}[thmvoid]{Lemma}
\newtheorem{proposition}[thmvoid]{Proposition}
\theoremstyle{remark}
\newtheorem{example}[thmvoid]{Example}
\newtheorem{question}[thmvoid]{Question}
\theoremstyle{definition}

\newtheorem{remark}[thmvoid]{Remark}

\numberwithin{equation}{section}

\newcommand{\Spec}{\mathrm{Spec}}

\newcommand{\Gal}{\mathrm{Gal}}

\newcommand{\Res}{\mathrm{Res}}

\newcommand{\XXX}{\mathscr{X}}

\newcommand{\ZZZ}{\mathscr{Z}}
\newcommand{\UUU}{\mathscr{U}}
\newcommand{\VVV}{\mathscr{V}}

\newcommand{\Aa}{\mathbb{A}}
\newcommand{\Qq}{\mathbb{Q}}

\newcommand{\Pp}{\mathbb{P}}
\newcommand{\R}{\Res_{L/K}}

\newcommand{\oK}{\overline{K}}

\title{Hilbert properties under base change in small extensions}

\author{Lior Bary-Soroker}
\address{School of Mathematical Sciences, Tel Aviv University,
Tel Aviv 69978, Israel}
\email{barylior@tauex.tau.ac.il}

\author{Arno Fehm}
\address{Institut f\"{u}r Algebra, Technische Universit\"{a}t Dresden, 01062 Dresden, Germany}
\email{arno.fehm@tu-dresden.de}

\author{Sebastian Petersen}
\address{Institut für Mathematik, Universit\"at Kassel, 34121 Kassel, Germany}
\email{petersen@mathematik.uni-kassel.de}

\begin{document}

\maketitle

\begin{abstract}
We study the preservation of the Hilbert property
and of the weak Hilbert property
under base change in field extensions.
In particular we show that these properties are preserved
if the extension is finitely generated or Galois with finitely generated Galois group,
and we also obtain some negative results.
\end{abstract}

\section{Introduction}

\noindent
The Hilbert property HP of Colliot-Th\'el\`ene--Sansuc and Serre \cite{Serre,CTS}
and the weak Hilbert property WHP introduced later by Corvaja--Zannier \cite{CZ}
are derived from Hilbert's irreducibility theorem and have been studied intensively in recent years, see for example 
the list of references in the introduction of \cite{BFP24}.

It is known that if a variety $X$ over a field $K$ of characteristic zero has HP respectively WHP, then so does the base change $X_L$,
for any finite field extension $L/K$.
The work \cite{BFP24} proves that 
in the case of $X=A$ an abelian variety over a number field,
the same holds true (for WHP) also in very specific infinite Galois extensions $L/K$.

The aim of this note is to provide a few more such base change results,
which are less arithmetic in nature but more general:

\begin{theorem}[cf.\ Cor.~\ref{cor:fg}]\label{int:FGthm}
Let $K$ be a field of characteristic zero,
$X$ a normal $K$-variety,
and $L/K$ a finitely generated field extension.
If $X$ has HP (respectively WHP), 
then so does $X_L$.
\end{theorem}

As an application we then give a new proof of the fact that the class of varieties with HP is closed under products (Corollary \ref{FT:cor}).

Our second base change result is for field extensions $L/K$ that are {\em small}
in the concrete sense that $L/K$ is algebraic and has only finitely many subextensions of degree $n$,
for every $n\in\mathbb{N}$.

\begin{theorem}[cf.\ Cor.~\ref{cor:small}]\label{int:smallthm} 
Let $K$ be a field of characteristic zero,
$X$ a normal $K$-variety
and $L/K$ a small extension.
If $X$ has HP (respectively WHP), then so does $X_L$.
\end{theorem}

In fact most of our results are stated for thin and strongly thin sets rather than HP and WHP, adding flexibility.
As opposed to many results in the literature we also try to avoid assuming our varieties are proper and smooth.

In the last section we discuss further questions regarding base change of HP and WHP
and obtain a few negative results.
In particular,
HP and WHP do not go {\em down} finite extensions
(in a strong sense, see Proposition \ref{prop:exQtr}),
and neither ``not HP'' nor ``not WHP'' is preserved in unions of chains, even if the extension is small (Proposition \ref{prop:smallex}).

\section{Finite extensions}

\noindent
We begin by fixing notation and definitions for this note.
Let $K$ always be a field of characteristic zero. 
A {\em $K$-variety} is a separated integral scheme of finite type over $K$. The degree $\deg(f)$ of a finite morphism $f\colon X\to Y$ of $K$-varieties is the degree of the associated function field extension.
A {\em cover} is a finite surjective morphism of normal $K$-varieties.
For a normal $K$-variety $X$,
a subset $T$ of $X(K)$ is {\em thin} (resp.\ {\em strongly thin}) in $X$
if $T\subseteq C(K) \cup \bigcup_{i=1}^r f_i(Z_i(K))$
for some proper closed subscheme $C$ of $X$ and a finite collection $(f_i\colon Z_i\rightarrow X)_{i=1,\dots,r}$ of covers
with ${\rm deg}(f_i)>1$ (resp.\ with $f_i$ ramified) for each $i$.
See \cite[\S2]{BFP24}
for basics on covers.
A normal\footnote{It makes sense to define HP also for non-normal varieties, but in order for us to treat HP and WHP simultaneously we restrict to normal ones.} $K$-variety $X$ has the Hilbert property HP  (resp.\ the weak Hilbert property WHP) if $X(K)$ is not thin (resp.\ not strongly thin) in $X$. 
We can and will always assume that the $Z_i$ are geometrically irreducible, since otherwise $Z_i(K)=\emptyset$.
Also note that every normal $K$-variety with WHP is automatically geometrically irreducible.

The following lemma,
which like many of our results in this work treats thin and strongly thin sets simultaneously, should be considered known to experts, and we include the proof only for completeness:
For thin sets, this is precisely \cite[Prop.~3.2.1]{Serre},
and for strongly thin sets,
\cite[Prop. 3.15]{CDJLZ} gives a proof for the case $X$ smooth proper and $T=X(L)$, which we follow closely.

\begin{lemma}\label{FG:fin}
Let $X$ be a normal $K$-variety,
let $L/K$ be a finite extension
and let $T\subseteq X(L)$. 
If $T$ is thin (resp.\ strongly thin) in $X_L$, then $T\cap X(K)$ is thin (resp.\ strongly thin) in $X$.  
\end{lemma}

\begin{proof}
It suffices to consider the cases $T=C(L)$ and $T=f(Z(L))$
for $C$ a closed proper subscheme of $X_L$ and $f\colon Z\to X_L$ a cover with $\deg(f)>1$ (resp.\ a ramified cover)
and $Z$ geometrically irreducible. 
As the projection $p\colon X_L\to X$ is closed,
$C'=p(C)$ is a proper closed subset of $X$, and so $C(L)\cap X(K)\subseteq C'(K)$ is strongly thin in $X$. 
For the case $T=f(Z(L))$,
we may assume that 
$X$ is affine. 
Then the Weil restriction
$\Res_{L/K}(f)\colon \Res_{L/K}(Z)\to \Res_{L/K}(X_L)$ exists and is a cover
of $K$-varieties
(see~\cite[4.10.2]{Scheiderer} for all properties except for integrality,
which is spelled out in \cite[Lemma 1.10]{Diem}, and normality, which can be deduced from
\cite[Thm.~1.3.1]{Weil}).
With
$\Delta\colon X\to\Res_{L/K}(X_L)$
the  diagonal morphism (cf.~\cite[4.2.5]{Scheiderer}),
the cartesian square
$$\xymatrix{
 \R(Z)\ar[d]_{\R(f)} & V\ar[l]\ar[d]^g\\
\R(X_L) & X\ar[l]_-{\Delta}
}$$
gives rise to the following commutative diagram in the category of sets:
$$\xymatrix{
Z(L)\ar[d]_f & \R(Z)(K)\ar[l]_-{\cong}\ar[d]_{\R(f)} & V(K)\ar[l]\ar[d]^g\\
X(L) & \R(X_L)(K)\ar[l]_-{\cong} & X(K)\ar[l]
}$$
The right hand square is cartesian  and the composition of the lower horizontal maps is just the inclusion, 
hence $f(Z(L))\cap X(K)\subseteq g(V(K))$,
and so it suffices to prove that $g(V(K))$ is thin (resp.\ strongly thin) in $X$,
i.e.~that $g(W(K))$ is thin (resp.\ strongly thin) in $X$ for every irreducible component
$W$ of $V$. 
As $g$ is finite and therefore closed,
we can assume without loss of generality that $g|_W\colon W\to X$ is surjective.
The normalization $\pi\colon \overline{W}\to W$
is an isomorphism above a nonempty open subset $W'$ of $W$.
Then $h:=g\circ \pi\colon \overline{W}\to X$ is a cover
and $B:=g(W\setminus W')$ is a proper closed subset of $X$ with
$g(W(K))\subseteq h(\overline{W}(K))\cup B(K)$. 
Thus it is enough to show that $h(\overline{W}(K))$ is thin (resp.\ strongly thin) in $X$. 
For this we have to check that $\deg(h)>1$ (resp.\ $h$ is ramified).  
Again we can assume that the $K$-variety $\overline{W}$ is geometrically irreducible. 
The following diagram commutes,
where $p'$ and $p''$
denote the projections:
$$
 \xymatrix{
Z\ar[d]_f & \R(Z)_L\ar[l]_-{p'}\ar[d]_{\R(f)} & \overline{W}_L\ar[l]_-{\delta}\ar[d]^{h_L}\\
X_L & \R(X_L)_L\ar[l]_-{p''} & X_L\ar[l]_-{\Delta_L}
}
$$
As $p''\circ \Delta_L={\rm id}_{X_L}$ (see again \cite[4.2.5]{Scheiderer}), the map 
$\alpha:=p'\circ\delta\colon\overline{W}_L\to Z$ is a cover
with $h_L=f\circ \alpha$.
This implies $\deg(h)=\deg(h_L)\ge \deg(f)>1$. 
If $f$ is ramified, then so is $h_L$ (cf.\ \cite[Lemma 2.15]{BFP24}) and thus also $h$.
\end{proof}

\section{Finitely generated extensions}

\noindent
In this section we discuss HP and WHP -- and more precisely thin and strongly thin sets -- under base change in finitely generated extensions.

\begin{lemma}  \label{FG:pure}
Let $S$ be a $K$-variety
with $S(K)$ dense in $S$,
and write $L=K(S)$. 
Let $X$ be a normal $K$-variety,
and let $T\subseteq X(L)$. 
If $T$ is thin (resp.\ strongly thin) in $X_L$, then $T\cap X(K)$ is thin (resp.\ strongly thin) in $X$.  
\end{lemma}

\begin{proof} 
It suffices to consider the case
$T=C(L)\cup f(Z(L))$
with $C$ a closed proper subscheme of $X_L$ and $f\colon Z\to X_L$ a cover with $\deg(f)>1$ (resp.\ a ramified cover) and $Z$ geometrically irreducible. 
There is a nonempty open subscheme $U$ of $X_L\setminus C$ such that 
$V:=f^{-1}(U)\to U$ is \'etale.  

After replacing $S$ by a suitable nonempty open subscheme we can assume that
$S$ is normal,
and that $Z$ and $U$ extend to $S$-schemes $\ZZZ$ and $\UUU$ 
of finite type (cf.\ \cite[IV.8.8.2]{EGAIV3}) in such a way that the morphisms $\ZZZ\to S$ and $\UUU\to S$ are flat, separated and surjective (cf.\ \cite[IV.8.9.4, IV.8.10.5]{EGAIV3}).
Moreover, as $Z$ and $U$ are geometrically irreducible normal varieties,
we can assume that all fibers of $\ZZZ\to S$ and of $\UUU\to S$ are geometrically irreducible normal varieties (cf.\ \cite[IV.9.9.5, IV.9.7.7]{EGAIV3}, 
\cite[III.9.1.13]{EGAIII1}).

After replacing $S$ by one of its nonempty open subschemes and replacing the rest accordingly, 
$f\colon Z\to X_L$ extends to a finite surjective $S$-morphism $F\colon \ZZZ\to X_S$, 
and the inclusion $U\to X_L$ extends to an open $S$-immersion $\UUU\to X_S$
(cf.\ \cite[IV.8.8.2, IV.8.10.5]{EGAIV3}). 
Then $V$ is the generic fiber of $\VVV:=F^{-1}(\UUU)$. For every $s\in S$ we have the following diagrams 
with cartesian squares.

$$\xymatrix{Z \ar[r]^{f}\ar[d] & X_{L} \ar[r]\ar[d] & \Spec(L)\ar[d] &
V \ar[r]^{f|_{V}}\ar[d] & U \ar[r]\ar[d] & \Spec(L)\ar[d]\\
\ZZZ \ar[r]^{F} & X_S \ar[r] & S &
\VVV \ar[r]^{F|_{\VVV}} & \UUU \ar[r] & S\\
\ZZZ_{s}\ar[u]\ar[r]^{F_{s}} & X_{K(s)}\ar[r]\ar[u] & \Spec(K(s))\ar[u]&
\VVV_{s}\ar[u]\ar[r]^{F_{s}|_{\VVV_{s}}} & \UUU_s \ar[r]\ar[u] & \Spec(K(s))\ar[u]}$$

Since $f|_V$ is \'etale,
after shrinking $S$ once more we can assume that for every $s\in S$ the morphism
$F_{s}|_{\VVV_{s}}\colon \VVV_{s}\to \UUU_s$
is \'etale and $\deg(F_s|_{\VVV_s})={\rm deg}(f)$ (cf.\ \cite[IV.17.7.11]{EGAIV4}). 
As $\ZZZ$ is flat over $S$, it follows that the morphism $F|_{\VVV}\colon \VVV\to \UUU$ is \'etale
(cf.\ \cite[IV.17.8.2]{EGAIV4}). 
Finally, if $f$ is ramified,
we can assume that
$F_{s}$ is ramified for every $s\in S$
\cite[IV.17.7.11]{EGAIV4}.

As $S(K)$ is dense in $S$ there exists $s\in S(K)$, 
and 
$$
 T_s:=(X(K)\setminus \UUU_s(K))\cup F_s(\ZZZ_s(K))
$$ 
is thin (resp.~strongly thin) in $X$,
so it suffices to prove that $T\cap X(K)\subseteq T_s$. 
Let $x\in T\cap X(K)$. 
If $x\notin \UUU_s(K)$ we are done,
so assume $x\in \UUU_s(K)$. 
We prove that $x\in F_s(\ZZZ_s(K))$. 
Let
$$
 \hat{x}\colon S\to X_S=X\times_K S,\ t\mapsto (x,t)
$$ 
be the constant $S$-morphism corresponding to $x$. 
The set $\hat{x}^{-1}(\UUU)$ is open in $S$ and contains $s$. 
We replace $S$ by $\hat{x}^{-1}(\UUU)$ and can thus view $\hat{x}$ as a section 
$\hat{x}\colon S\to \UUU$ of $\UUU\to S$.
Consider the cartesian square
$$\xymatrix{ F^{-1}(\hat{x})\ar[r]\ar[d] & S\ar[d]^{\hat{x}}\\ 
\VVV\ar[r] & \UUU}$$
and note that $F^{-1}(\hat{x})$ is a finite \'etale $S$-scheme whose generic fiber is $f^{-1}(x)$ and
whose fiber over $s$ is $F_{s}^{-1}(x)$. In particular,
every connected component of $F^{-1}(\hat{x})$ is an \'etale cover of $S$.
So since
$f^{-1}(x)\to \Spec(L)$ has a section by the assumption
$x\in T$,
also $F^{-1}(\hat{x})\to S$ has a section, 
and this implies that  $F_{s}^{-1}(x)\to \Spec(K)$ has a section. It follows that $x\in F_s(\ZZZ_s(K))$ as desired. 
\end{proof}

\begin{proposition}\label{FG:main} Let $X$ be a normal $K$-variety, $L/K$ a finitely generated field extension 
and $T\subseteq X(L)$. 
If $T$ is thin (resp.\ strongly thin) in $X_L$, then $T\cap X(K)$ is thin (resp.\ strongly thin) in $X$. 
\end{proposition}

\begin{proof} 
Fix an intermediate field $K\subseteq F\subseteq L$ such that $F/K$ is purely transcendental, i.e.~$F=K(\mathbb{A}^n)$ for some $n$, and $L/F$ is algebraic, hence finite.
The claim then follows from
Lemma \ref{FG:fin} applied to $L/F$
and
Lemma \ref{FG:pure} 
applied to $S=\mathbb{A}^n$.
\end{proof}

\begin{corollary}\label{cor:fg}
Let $X$ be a normal $K$-variety
and $L/K$ a finitely generated field extension.
If $X$ has HP (resp.\ WHP), then $X_L$ has HP (resp.\ WHP).     
\end{corollary}

In \cite{Serre}, Serre had asked whether the product of two varieties with HP has HP again.
This was positively answered in \cite{BFP}
by proving and applying a fibration theorem for HP.
We now 
give another proof of this product theorem
by first proving a variant of the fibration theorem using the generic fiber and then combining this with Corollary \ref{cor:fg}.

\begin{proposition}\label{FT:main}
Let $p\colon \XXX\to S$ be a dominant morphism of normal $K$-varieties.
Assume that $S$ has HP 
and that the generic fiber $X$ of $p$ is normal and has HP. 
Then $\XXX$ has HP.
\end{proposition}

\begin{proof}
Let $\UUU$ be a nonempty open subset of $\XXX$, and let $(F_i\colon\ZZZ_i\to\XXX)_{i=1,\dots,r}$ be a family of covers with $\deg(F_i)>1$ 
and $\ZZZ_i$ geometrically irreducible for all $i$. 
We have to show that $\UUU(K)\setminus \bigcup_{i=1}^r F_i(\ZZZ_i(K))$ is nonempty.
After replacing $\UUU$ by a smaller nonempty open set we can assume that $f_i^{-1}(\UUU)\to \UUU$ is \'etale for all $i$. 
Thus, replacing $\XXX$ by $\UUU$, we can assume the $F_i$ \'etale right from the outset and have to prove that $\XXX(K)\setminus \bigcup_{i=1}^r F_i(\ZZZ_i(K))$ is not empty. 
For each $i$,
the generic fiber $f_{i}\colon Z_{i}\to X$ of $F_i$ is an \'etale cover of $X$ with $\deg(f_{i})=\deg(F_i)>1$.
As $X$ has HP, there exists a point 
$$
 x\in X(L)\setminus \bigcup_{i=1}^r f_{i}(Z_{i}(L)),
$$ 
where $L=K(S)$.
After replacing $S$ by one of its nonempty open subsets (and replacing the rest accordingly) $x$ extends to a section $\sigma\colon S\to\XXX$ of $p$. 
Form the cartesian square
$$
 \xymatrix{
\ZZZ_{i, \sigma}\ar[r]^{F_{i,\sigma}}\ar[d] & S\ar[d]^\sigma\\
\ZZZ_i \ar[r]^{F_i}& \XXX
}
$$
and note that the generic fiber of the finite \'etale morphism  $F_{i,\sigma}$ is $f_{i}^{-1}(x)$ and that $f_{i}^{-1}(x)\to\Spec(L)$
does not admit a section by our choice of $x$. 
Thus $F_{i,\sigma}$ does not admit a section. 
It follows that for every $i$ the connected components of $\ZZZ_{i,\sigma}$ 
are all of degree $>1$ over $S$. 
As $S$ has HP, there exists $s\in S(K)\setminus \bigcup_{i=1}^r F_{i, \sigma}(\ZZZ_{i, \sigma}(K))$. 
Thus $\sigma(s)\in\XXX(K)\setminus \bigcup_{i=1}^r F_i(\ZZZ_i(K))$, as desired.
\end{proof}

\begin{corollary}[{\cite[Cor.~3.4]{BFP}}]\label{FT:cor} 
Let $X$ and $Y$ be normal $K$-varieties. 
If both $X$ and $Y$ have HP, so does $X\times_K Y$.
\end{corollary}

\begin{proof}  
Since $X$ has HP,
by Corollary \ref{cor:fg}
also the generic fiber $X_{K(Y)}$ of
the projection $X\times_K Y\to Y$
has HP. 
Therefore $X\times_K Y$ has HP by Proposition \ref{FT:main}. 
\end{proof}

By \cite[Thm.~1.9]{CDJLZ},
Corollary \ref{FT:cor} holds for WHP 
(and $X,Y$ smooth proper $K$-varieties) instead of HP when $K$ is finitely generated, but the case of general $K$ is open.
Such a product theorem for WHP would follow from Corollary \ref{cor:fg} if one could prove Proposition \ref{FT:main} for WHP instead of HP.
Even the following weaker version seems open:

\begin{question}\label{q:generic}
In Proposition \ref{FT:main},
if we assume that $X$ has only WHP instead of HP,
does it still follow that $\XXX$ has WHP?
\end{question}

\section{Small extensions}

\noindent
In this section we discuss HP and WHP -- and again more precisely thin and strongly thin sets -- in small extensions as defined in the introduction.
A {\em Galois} extension is small if and only if its Galois group is {\em small} in the sense
that it has only finitely many open subgroups of index $n$ for every $n\in\mathbb{N}$,
and every (topologically) finitely generated profinite group is small.
See \cite[Chapter 16.10]{FJ} for a discussion of small profinite groups and their relation to finitely generated profinite groups.
An arithmetically important example of a small Galois extension
is the maximal Galois extension of a number field unramified outside a given finite set of primes, see \cite[Example 16.10.9]{FJ}.
For a combinatorial sufficient condition for a Galois extension to have small Galois group see \cite[Proposition 2.5]{BF_random}.
Small {\em absolute} Galois groups play an important role in many model theoretic results, see \cite{FehmJahnke}.

\begin{lemma} \label{SM:part}
Let $X$ be a normal $K$-variety and $f\colon Z\to X$ a cover of degree $n$.
Let $L/K$ be an algebraic extension
and let $L_0$ be the compositum of all subextensions of $L/K$ of degree at most $n$. 
Then $f(Z(L))\cap X(K)= f(Z(L_0))\cap X(K)$. 
\end{lemma}

\begin{proof} 
Since $X$ is normal, 
the finite surjective morphism $f$ is universally open \cite[14.4.4]{EGAIV3}. 
As ${\rm char}(K)=0$,
\cite[15.5.2]{EGAIV3} therefore gives for $x\in X(K)$ that 
$$
 \sum_{z\in f^{-1}(x)} [K(z):K]\le n,
$$
in particular $f^{-1}(x)(L)=f^{-1}(x)(L_0)$. 
Thus $x\in f(Z(L))$ if and only if $x\in f(Z(L_0))$.
\end{proof}

\begin{proposition} \label{SM:main}
Let $X$ be a normal $K$-variety, $L/K$ a small extension
and $T\subseteq X(L)$. 
If $T$ is thin (resp.\ strongly thin) in $X_L$, then $T\cap X(K)$ is thin (resp.\ strongly thin) in $X$. 
\end{proposition}

\begin{proof}
It suffices to consider the case 
$T=C(L)\cup f(Z(L))$
for a proper closed subset $C$ of $X_L$ and $f\colon Z\to X_L$ 
a cover of degree $n>1$ (resp.\ a ramified cover).
Since $L/K$ is algebraic, there exists a finite extension $K'/K$ in $L$, a 
proper closed subset $C'$ of $X_{K'}$ and a 
cover $f'\colon Z'\to X_{K'}$ of degree $n$ (resp.\ a ramified cover $f'\colon Z'\to X_{K'}$) such that 
$C'_L=C$, $Z'_L=Z$ and $f'_L=f$.  
We will now show that 
$$
 T':=T\cap X(K')=C'(K')\cup (f'(Z'(L))\cap X(K'))
$$
is thin (resp.\ strongly thin) in $X_{K'}$,
which by Lemma \ref{FG:fin} will prove that 
$T\cap X(K)=T'\cap X(K)$ is thin (resp.\ strongly thin) in $X$.
Clearly 
$C'(K')$ is strongly thin in $X_{K'}$,
and Lemma \ref{SM:part} applied to $f'$ gives that
$$
  f'(Z'(L))\cap X(K')= f'(Z'(L_0))\cap X(K'),
$$
where $L_0$ is the compositum of all subextensions of $L/K'$ of degree at most $n$.
As $Z'_{L_0}$ is irreducible (because even $Z'_L$ is irreducible), and in particular $f'_{L_0}$ is a cover of degree $n$ (resp.\ a ramified cover, cf.\ \cite[IV.2.7.1]{EGAIV2}), we have that $f'(Z'(L_0))$ is thin (resp.\ strongly thin) in $X_{L_0}$. 
As the extension $L_0/K'$ is finite because
$L/K$ is small,
this implies that
$f'(Z'(L_0))\cap X(K')$ is thin (resp.\ strongly thin) in $X_{K'}$ by Lemma \ref{FG:fin}. 
Hence $T'$ is thin (resp.\ strongly thin) in $X_{K'}$. 
\end{proof}

\begin{corollary}\label{cor:small}
Let $X$ be a normal $K$-variety and $L/K$ a small extension.
If $X$ has HP (resp.\ WHP), then $X_L$ has HP (resp.\ WHP).     
\end{corollary}

The case $X=\mathbb{P}^1$ and $L/K$ Galois of
this corollary appears already as \cite[Prop.~16.11.1]{FJ}.
The proof, however, uses irreducible specializations and therefore 
does not carry over from HP to WHP,
which is why the introduction of \cite{BFP24}
stated (in the special case of Galois extensions with finitely generated Galois group) 
that at that point we did not know how to prove it for WHP.

\begin{example}
For example, if $L$ denotes the cyclotomic $\mathbb{Z}_p$-extension of $\mathbb{Q}$,
and $E$ is an elliptic curve over $\mathbb{Q}$ of positive rank,
then $E$ has WHP
(as follows from Falting's theorem, see \cite{CZ}), 
and thus by Corollary \ref{cor:small}
so has $E_L$.
By \cite{Kato} (see also \cite[Example 1.2]{MazurRubin}), $E(L)$ is finitely generated.
This is in stark contrast with the results from \cite{BFP24},
where, 
for abelian extensions $M/\mathbb{Q}$, 
in order to prove that
$E_M$ has WHP
we needed to assume
that $E(M)$ has infinite rank,
in fact even that $E(M)/E(K)$ has positive rank for every number field $K\subseteq M$.
In particular, the results of \cite{BFP24} do not give that $E_L$ has WHP.
\end{example}

\section{Counterexamples}

\noindent
It is obvious that neither HP nor WHP is preserved in unions of chains of fields.
In this final section we discuss two other potential preservation theorems for these two properties and their negation.
Both constructions will make use of the following result for fields $K$ that are 
Hilbertian (i.e.~$\mathbb{P}^1_K$ has HP) and
pseudo-algebraically closed (PAC, see e.g.~\cite[Chapter 11]{FJ}):

\begin{proposition}\label{prop:hPAC}
Let $K$ be a Hilbertian PAC field.
Then every normal $K$-variety $X$ has HP.
\end{proposition}

\begin{proof}
By a theorem of Fried--Völklein \cite{FV},
the absolute Galois group of a countable Hilbertian PAC field is a free profinite group on countably many generators,
in particular $K$ is what is called $\omega$-free,
and \cite[Prop.~27.3.4]{FJ} gives
that over an $\omega$-free PAC field every variety has HP.
\end{proof}

Also the following Proposition \ref{WHPimpliesHilbertian},
which is a strengthening of
\cite[p.~20 Exercise 1]{Serre},
will be applied to both our examples:

\begin{lemma}\label{Aut1} 
Let $B,C\subseteq\Pp^1(\oK)$ be finite sets. There exists $\tau\in{\rm Aut}(\Pp^1_\Qq)$ with 
$\tau(B)\cap C=\emptyset$. 
\end{lemma}

\begin{proof}  
Without loss of generality, $B,C\subseteq\mathbb{A}^1(\oK)$.
Choose $a\in\mathbb{Q}\setminus(B-C)$ and let $\tau$ be the translation by $a$.
\end{proof}

\begin{lemma} \label{thinpullback} 
Let $g\colon X\to Y$ be a smooth surjective morphism
of normal $K$-varieties,
and let be $T\subseteq Y(K)$ strongly thin in $Y$. 
If
\begin{enumerate}[ (a)]
\item  the generic fiber of $g$ is geometrically irreducible, \underline{or}
\item $g$ is proper,
\end{enumerate}
then $\bar{T}:=\{x\in X(K): g(x)\in T\}$
is strongly thin in $X$.
\end{lemma} 

\begin{proof} 
If $T=C(K)$ for some proper closed subset $C$ of $Y$, then $\bar{C}:=g^{-1}(C)$ is a proper closed subset of $X$ containing $\bar{T}$, thus $\bar{T}$ is strongly thin. 
It hence suffices to consider the case where $T=f(Z(K))$ for a ramified cover $f\colon Z\to Y$ with $Z$ geometrically irreducible.
In the cartesian square
$$
 \xymatrix{Z\times_Y X \ar[r]^-{g'}\ar[d]^{f'} & Z\ar[d]^{f}\\ 
X\ar[r]^{g} & Y}
$$
the fiber product $Z\times_Y X$ is normal because $g'$ is smooth and $Z$ is normal. 
As $\bar{T}$ is contained in $f'((Z\times_Y X)(K))$, it suffices to prove that the latter set is strongly thin in $X$ in both cases (a) and (b). 

First assume that the generic fiber $g^{-1}(\eta)$ of $g$ is geometrically irreducible.
Then if $\zeta$ denotes the generic point of $Z$,
$(g')^{-1}(\zeta)=g^{-1}(\eta)\times_{k(\eta)} k(\zeta)$ is irreducible.
Thus as
$(g')^{-1}(\zeta)\to Z\times_Y X$ is dominant, because it is the pullback of the dominant map $\Spec(k(\zeta))\to Z$ by the surjective morphism $g'$, we get that $Z\times_Y X$ is irreducible and thus $f'$ is a cover. Since $f$ is ramified and $g$ is surjective, it follows that $f'$ is ramified \cite[17.7.3(i)]{EGAIV4}.
Thus $f'((Z\times_Y X)(K))$ is strongly thin in $X$. 

Now assume instead that $g$ is proper. 
Let $Z_1,\dots, Z_r$ be the irreducible components of $Z\times_Y X$, and denote by
$f_j'\colon Z_j\to X$ (resp.\ $g_j'\colon Z_j\to Z$) 
the restrictions of $f'$ (resp.\ $g'$) to $Z_j$. 
Each $Z_j$ is open in $Z\times_Y X$ because $Z\times_Y X$ is normal. 
We have to show that $\bar{T}_j:=f_j'(Z_j(K))$ is strongly thin in $X$ for every $j$. 
As $f_j'$ is finite, hence closed,
we can assume that $f_j'$ is surjective,
and therefore a cover.
It is enough to show that $f_j'$ is ramified.
Since $g'$ is smooth and proper, hence open and closed, and $Z$ is connected, $g_j'\colon Z_j\to Z$ is surjective and smooth. 
In particular $g_j'$ is flat (cf.\ \cite[17.5.1(i)]{EGAIV4}).
As $f$ is not smooth it follows that $f\circ g_j'$ is not smooth
(cf.\ \cite[17.7.7(i)]{EGAIV4}).
Hence $g\circ f_j'$ is not smooth and $f_j'$ cannot be \'etale, as desired.
\end{proof}

\begin{remark}
In the special case where $X$ and $Y$ are smooth and proper over $K$, and $T=Y(K)$, the previous lemma is implied by \cite[Thm.~3.7]{CDJLZ}; we need it in the non-proper case however. 
By the proof of \cite[Prop.~7.13]{CTS}, Lemma~\ref{thinpullback}(a)
holds for ``thin'' instead of ``strongly thin''. Lemma \ref{thinpullback}(b) though does not hold for ``thin'' instead of ``strongly thin'':  
The set $T=\Qq^{\times 2}$ is thin in $X:=\mathbb{G}_{m,\Qq}$,  the morphism $g\colon X\to X, x\mapsto x^2$ is smooth, proper and surjective, but $\{x\in X(\Qq):
g(x)\in T\}=\Qq^\times$ is not thin in $X$. 
\end{remark}

\begin{proposition}\label{WHPimpliesHilbertian}
If some normal $K$-variety $X$ has WHP,
then $K$ is Hilbertian.
\end{proposition}

\begin{proof} 
Let $(f_j\colon Z_j\to \Pp^1_K)_{j=1,\dots, r}$ be a family of covers with $\deg(f_j)>1$ and $Z_j$ geometrically irreducible for every $j$. Let $T=\bigcup_{j=1}^r f_j(Z_j(K))$,
and let $U$ be a nonempty open subset of $\Pp^1_K$. 
We have to show that $U(K)\setminus T$ is not empty. 
By the Hurwitz formula, 
$f_j$ is ramified for every $j$. 
In particular, $T$ is strongly thin in $\mathbb{P}^1_K$. 

As ${\rm char}(K)=0$, there exists a nonempty open subset $X'$ of $X$ and an \'etale cover $p\colon X'\to V$ where $V$ is some nonempty open subscheme of $\Aa^n_K$, $n=\dim(X)$. 
Let $q\colon V\to \Aa_K^1\subseteq \Pp^1_K$ be the projection on the first coordinate followed by the inclusion. 
Then $(q\circ p)(X')$ is open in $\Pp^1$. 
By Lemma \ref{Aut1} there exists an automorphism $\tau$ of $\Pp^1_K$ such that $\tau(q(p(X')))$ contains 
the finitely many branch points of each $f_j$. 
Let $g=\tau\circ q\circ p$. 
Note that $\tau\circ q$ is smooth and has geometrically connected fibers, and 
recall that $p$ is an \'etale cover, in particular smooth and proper. By Lemma~\ref{thinpullback} we conclude that $\bar{T}:=\{x\in X'(K): g(x)\in T\}$ is strongly thin in $X'$. It follows that $\bar{T}$ is strongly thin in $X$. 
Since $X$ has WHP,
$g^{-1}(U)(K)\setminus \bar{T}$ contains a point $x_0$,
and $g(x_0)\in U(K)\setminus T$ as desired.  
\end{proof}

Our first example shows that HP and WHP do not go {\em down} finite extensions.
More precisely, ``not HP'' and ``not WHP'' are not preserved under base change in finite extensions $L/K$, 
in a strong sense:

\begin{proposition}\label{prop:exQtr}
Let $K=\Qq^{\rm tr}$ be the maximal totally real Galois extension of $\Qq$ and let $L=\Qq^{\rm tr}(i)$. 
Then every normal $L$-variety $X$ has HP,
but no normal $K$-variety has WHP.
\end{proposition}

\begin{proof} 
The first claim follows
from Proposition \ref{prop:hPAC}
since $L$ is Hilbertian 
by a theorem of Weissauer, 
and PAC by results of Moret-Bailly and Green--Pop--Roquette,
see \cite[Example 5.10.7]{patching}.
The second claim follows from 
Proposition \ref{WHPimpliesHilbertian}
since $K$ is not Hilbertian,
see e.g.~\cite[Thm.~1]{BF}.
\end{proof}

Our second example is motivated by attempts to prove a general product theorem for WHP 
(see discussion before Question \ref{q:generic})
at least over small extensions $L$ of $\mathbb{Q}$:
If $X,Y$ are $L$-varieties with WHP,
one could try to descend $X$ and $Y$ to varieties with WHP over a number field $K\subseteq L$, apply the product theorem of \cite{CDJLZ} there, and then use Corollary \ref{cor:small}
for the small extension $L/K$ to conclude that 
$X\times_L Y$ has WHP. 
The following Proposition \ref{prop:smallex} shows that this strategy fails.

\begin{lemma}\label{lem:smallPAC}
    Let $K$ be a countable Hilbertian field. 
    There exists a PAC field that is a small extension of $K$.
\end{lemma}

\begin{proof}
    Let $(C_i)_{i\in\mathbb{N}}$ be an enumeration of the geometrically integral $K$-curves. 
    For each $i$ we apply the stabilizing basis theorem \cite[Theorem 18.9.3]{FJ}     
    to obtain an 
    integer $n_i$ and a 
    non-constant rational map $f_i\colon C_i\to \mathbb{P}^1$ of degree $n_i$ with the following properties:
    \begin{enumerate}
        \item \label{coolaid} $5\leq n_1< n_2< n_3< \ldots$, and 
        \item The monodromy group of $f_i$ (that is, the Galois group of the Galois closure of $K(C_i)/K(\mathbb{P}^1)$ viewed as a permutation group acting on the generic fiber) and the geometric monodromy group of $f_i$ (i.e., the monodromy group over $\overline{K}$) are both $S_{n_i}$, for each $i$.
    \end{enumerate}
    As in the proof of \cite[Lemma 16.2.6]{FJ},
    since $K$ is Hilbertian, we may inductively construct  $a_i\in \mathbb{P}^{1}(K)$ such that, for each $i$, 
    \begin{enumerate}\addtocounter{enumi}{2}
        \item The fiber $f_i^{-1}(a_i)$ of $f_i$ above $a_i$ is integral. In particular, $f_i^{-1}(a_i)=\Spec(E_i)$ for some field extension $E_i/K$ of degree $n_i$. (In the notation of \cite[Lemma 16.2.6]{FJ}, $E_i=K(b_i)$.)
        \item If $L_i$ denotes the Galois closure of $E_i/K$, we have $\Gal(L_i/K)=S_{n_i}$ and $L_i$ is linearly disjoint over $K$ from the compositum $L_1\cdots L_{i-1}$. 
    \end{enumerate}
    Let $E=\prod_{i=1}^\infty E_i$ and $L=\prod_{i=1}^\infty L_i$ be the field composita. 
    By construction, $C_i(E)\neq\emptyset$ for every $i$, hence $E$ is PAC \cite[Theorem 11.2.3]{FJ}. 
    Also, by construction, $L/K$ is Galois, $\Gal(L/K) = \Gamma:= \prod_{i=1}^{\infty} S_{n_i}$, and $\Gal(L/E) = \Gamma_1:=\prod_{i=1}^{\infty} S_{n_i-1}$, where $S_{n-1}\leq S_n$ is the stabilizer of a point.
    Let $\Delta:= \prod_{i=1}^\infty A_{n_i}\leq\Gamma$. 
    Since the $A_{n_i}$ are distinct, simple non-abelian, and generated by two elements,  $\Delta$ is topologically generated by two elements, so in particular the number $N_r(\Delta) $ of open subgroups of $\Delta$ of index $r$ is finite, for every $r\in\mathbb{N}$.
    On the other hand, since $n_i\geq5\geq 3$, $S_{n_i-1}\cdot A_{n_i} = S_{n_i}$  (where we denote $H_1\cdot H_2 = \{ h_1 h_2 : h_1\in H_1, \ h_2\in H_2\}$). 
    So $\Gamma_1\cdot\Delta = \Gamma$. 
    
    Let $\mathcal{G}$ be the set of open subgroups of $\Gamma$ containing $\Gamma_1$. 
    For $G\in\mathcal{G}$, the inclusion $\Delta\rightarrow\Gamma$
    induces a bijection $\Delta/(G\cap\Delta)\rightarrow\Gamma/G$,
    hence $(\Gamma:G)=(\Delta:G\cap \Delta)$.
    In particular, the map $\mathcal{G} \ni G\mapsto G\cap \Delta$ is injective,
    which implies that there are at most $N_r(\Delta)$ many open subgroups of $\Gamma$ of index $r$ containing $\Gamma_1$. 
    So by the Galois correspondence, there at most $N_r(\Delta)$ subextensions of $E/K$ of degree $r$, for every $r$, hence $E/K$ is small. 
\end{proof}

\begin{proposition}\label{prop:smallex}
There exists a small extension $L/\mathbb{Q}$ 
and a smooth proper $\mathbb{Q}$-variety $X$
such that $X_L$ has HP but 
$X_K$ does not have WHP for any number field $K$.
\end{proposition}

\begin{proof}
Let $X$ be any smooth proper $\mathbb{Q}$-curve of genus at least $2$,
e.g.~the Fermat curve $x^4+y^4=z^4$.
By Lemma \ref{lem:smallPAC}
there exists a small extension $L$ of $\mathbb{Q}$ which is PAC.
As $L/\mathbb{Q}$ is small and $\mathbb{Q}$ is Hilbertian,
so is $L$ (cf.~Proposition \ref{cor:small}).
So by Proposition \ref{prop:hPAC}, 
$X_L$ has HP.
By Falting's theorem, 
$X(K)$ is not Zariski-dense in $X$, 
so $X_{K}$ does not have WHP,
for any number field $K$.
\end{proof}

In particular, ``not HP'' and ``not WHP'' do not go up in chains
of number fields even when the union of the chain is a small extension of $\mathbb{Q}$.
We remark that the extension $L/\mathbb{Q}$
constructed here has Galois closure which is not small over $\mathbb{Q}$.

\begin{question}
Does there exist a small {\em Galois} extension $L/\mathbb{Q}$ with $L$ PAC?
\end{question}

\section*{Acknowledgements}

\noindent
A.F.~would like to thank Philip Dittmann for helpful discussions around Proposition \ref{WHPimpliesHilbertian}.
L.B.-S. was supported by the Israel Science Foundation (grant no.~702/19).
L.B.-S. and S.P. thank TU Dresden for hospitality during a research visit in December 2023 where part of this work was done.

\end{document}